\documentclass[11pt,oneside]{amsart}

\newcommand\mtop{1in}
\newcommand\mbottom{1in}
\newcommand\mleft{1in}
\newcommand\mright{1in}
\usepackage[top = \mtop, bottom = \mbottom, left = \mleft, right=\mright]{geometry}

\usepackage{anysize}
\usepackage{amsmath}
\usepackage{amsthm}
\usepackage{amssymb}
\usepackage{stmaryrd}
\usepackage{wasysym}
\usepackage{mathrsfs}
\usepackage{hyperref}
\usepackage{dsfont}
\usepackage{soul}
\hypersetup{colorlinks=true,linkcolor=blue,citecolor=magenta}
\usepackage{color}
\usepackage[all]{xy}
\usepackage{float}
\usepackage{bm}
\usepackage{mathtools}
\usepackage{comment}
\usepackage{thmtools}
\usepackage{setspace}

\newtheorem{Definition}{Definition}

\newtheorem*{MainTheorem*}{Main Theorem}
\theoremstyle{definition}
\newtheorem{Example}{Example}[section]
\newtheorem{Theorem}{Theorem}[section]
\newtheorem{Prop}{Proposition}[section]

\newtheorem{Lemma}{Lemma}[section]

\DeclareMathOperator{\sgn}{sgn}

\renewcommand \l {\lambda}
\renewcommand \r {\rho}

\newcommand*{\exone}[1]{{#1}_{\widehat{1}}}
\newcommand*{\extwo}[1]{{#1}_{\widehat{1},\widehat{2}}}

\newcommand*{\zetatwo}[2]{\zeta_{#1,#2}}

\newcommand*{\HL}[1]{R_{#1}(x;t)}
\newcommand*{\WD}[1]{\Delta_{#1}}
\newcommand*{\defWD}[2]{\Delta_{#1}(#2)}

\newcommand*{\MHL}[2]{P_{#1}(#2)}
\newcommand*{\HLtwo}[2]{R_{#1}(#2)}
\newcommand*{\R}{{\phi}}
\newcommand*{\conc}[2]{{#1} \| {#2}}
\newcommand*{\GT}{SGT} 
\newcommand*{\G}{GT}
\newcommand*{\tokbij}{{\Theta}}
\newcommand*{\thetaset}{{\Omega}}
\newcommand*{\GTtwo}{{GT_2}}
\newcommand*{\row}[1]{{r_{#1}}}
\renewcommand*{\a}{\alpha}
\renewcommand*{\tau}{a}
\newcommand*{\deltaone}[2]{\delta_{#1}{(x_{#2})}}
\newcommand*{\deltatwo}[3]{\delta_{{#1},{#2}}{(x_{#3})}}
\newcommand*{\deltathree}[4]{\delta_{#1,#2,#3}{(x_{#4})}}

\newcommand*{\defdeltaone}[3]{\delta_{#1}{(x_{#2};#3)}}
\newcommand*{\defdeltatwo}[4]{\delta_{{#1},{#2}}{(x_{#3};#4)}}
\newcommand*{\defdeltathree}[5]{\delta_{#1,#2,#3}{(x_{#4};#5)}}
\newcommand*{\defProdone}[2]{\frac{\defdeltaone{#1}{#2}{t}}{\deltaone{#1}{#2}}}
\newcommand*{\defProdtwo}[3]{\frac{\defdeltatwo{#1}{#2}{#3}{t}}{\deltatwo{#1}{#2}{#3}}}
\newcommand*{\defProdthree}[4]{\frac{\defdeltathree{#1}{#2}{#3}{#4}{t}}{\deltathree{#1}{#2}{#3}{#4}}}
\newcommand*{\M}[2]{  M(#1;#2) }

\title{A Generalization of Tokuyama's Formula to the Hall-Littlewood Polynomials} 

\author{Vineet Gupta}
\address[Vineet Gupta]{Department of Mathematics, Stanford University, Stanford, CA 94305}
\email{vineetg@stanford.edu}

\author{Uma Roy}
\address[Uma Roy]{Department of Mathematics, Boston University, Boston, MA 02215}
\email{uma.roy.us@gmail.com}

\author{Roger Van Peski}
\address[Roger Van Peski]{Department of Mathematics, Princeton University, Princeton, NJ 08540}
\email{rpeski@princeton.edu}

\begin{document}

\maketitle

\begin{abstract}
A theorem due to Tokuyama expresses Schur polynomials in terms of Gelfand-Tsetlin patterns, providing a deformation of the Weyl character formula and two other classical results, Stanley's formula for the Schur $q$-polynomials and Gelfand's parametrization for the Schur polynomial. We generalize Tokuyama's formula to the Hall-Littlewood polynomials by extending Tokuyama's statistics. Our result, in addition to specializing to Tokuyama's result and the aforementioned classical results, also yields connections to the monomial symmetric function and a new deformation of Stanley's formula.

\end{abstract}

\section{Introduction}

\indent Schur polynomials, a special class of symmetric polynomials, play an important role in representation theory. They encode the characters of irreducible representations of general linear groups, which may be computed via the Weyl character formula. Tokuyama \cite{tokuyama} gave a deformation of the Weyl character formula for $\text{GL}_n(\mathbb{C})$ (Cartan type $A_n$). This formula expresses Schur polynomials in terms of statistics obtained from strict Gelfand-Tsetlin (GT) patterns and includes two other classical results as specializations, the Gelfand parameterization formula for Schur polynomials \cite{gelfand} and Stanley's formula for the Schur $q$-polynomials \cite{stanley}.

    The ideas in \cite{tokuyama} have been extended to other Cartan types. For example, combinatorial expressions of deformations of the Weyl denominator for Cartan types $B_n$, $C_n$, and $D_n$ were given by Okada \cite{okadadenom} and Simpson \cite{simpson1}, \cite{simpson2}. Hamel and King in \cite{hamel02} replicate Tokuyama's deformation of the Weyl character formula in type $C_n$, and Friedberg and Zhang \cite{friedberg2014tokuyama} derived a similar result for type $B_n$. These results are also often expressed using other combinatorial objects such as Young tableaux \cite{fulton1997young}, alternating sign matrices \cite{tokuyama}, \cite{okadadenom}, \cite{hamel2007bijective}, and 6-vertex or ice-type models \cite[Chapter 19]{brubaker2006weyl}, \cite{brubaker2011schur}, \cite{tabony2011deformations}. Hamel and King express the type $A_n$ case \cite{hamel2007bijective} and type $C_n$ case \cite{hamel02} using 6-vertex partition functions, and Brubaker and Schultz \cite{brubaker20146} give Tokuyama-type deformations for types $A_n$, $B_n$, $C_n$, and $D_n$ using modified ice models. One can also try to generalize Tokuyama's ideas to other symmetric polynomials, such as the Hall-Littlewood polynomials, and this is the problem we consider.
    
    The Hall-Littlewood polynomials are a class of symmetric polynomials which may be viewed as a generalization of the Schur polynomials by a deformation along a parameter $t$. The Hall-Littlewood polynomials also interpolate between the dual bases of the Schur polynomials and the monomial symmetric functions at $t=0$ and $t=1$, respectively.  These polynomials are used to determine characters of Chevalley groups over local and finite fields \cite{tokuyama}. Stanley's formula expresses the Hall-Littlewood polynomials at the singular value $t=-1$ (commonly known as the Schur $q$-polynomials \cite[Chapter III]{macdonaldbook}) as a summation over strict GT patterns. However, there does not exist an analogue of Tokuyama's formula expressing the Hall-Littlewood polynomials as a summation over combinatorial statistics from Gelfand-Tsetlin patterns. In this paper we provide such a result. Theorem \ref{MainTheorem2*}, in addition to linking the classical specializations of Tokuyama, reduces to a different deformation of Stanley's formula at $t=-1$, and a formula for the monomial symmetric functions in terms of GT patterns at $t=1$.

\section{Preliminary Notation and a Theorem due to Tokuyama}\label{prelim}

A partition $\l = (\l_1,\l_2,\ldots,\l_n)$ is a finite tuple of nonnegative integers, referred to as \emph{parts}. Unless otherwise stated, a partition will be assumed to be weakly decreasing, i.e.  $\l_i \geq \l_{i+1}$ for all $i$. The \emph{length} of a partition $\l$ is the number of parts in $\l$, and the \emph{size} of $\l$ is defined as $|\l|=\sum_{i=1}^n \lambda_i$. Addition of partitions of equal length is done component-wise, and given two partitions $\l$ and $\mu$ of lengths $n$ and $m$ respectively, we express their concatenation as the tuple $\conc{\l}{\mu} = (\l_1,\ldots,\l_n,\mu_1,\ldots,\mu_m)$. We will typically use $\a$ to denote some strictly decreasing partition, often taking $\a = \lambda + \rho$ when $\lambda$ is defined. 

\noindent Define the partition $\r_n$ as 
\begin{equation}
\rho_n=(n-1,n-2,\ldots,1,0).
\end{equation}
We often write $\r$ in place of $\r_n$ as the value of $n$ is clear from context. 

We write the polynomial $f(x)$ as short for $f(x_1,\ldots, x_n)$, and similarly $x^\lambda = x_1^{\lambda_1}x_2^{\l_2} \ldots  x_n^{\lambda_n}$. Furthermore, a permutation $\sigma \in S_n$ acts on $f(x)$ by permuting the variables $x_i$.  

\noindent The Weyl denominator $\WD{n}$ is given by the formula
\begin{equation} 
\WD{n} = \prod_{1 \leq i<j \leq n} (x_i-x_j).
\end{equation} 

A deformation of the Weyl denominator $\defWD{n}{t}$ is given by the similar formula
\begin{equation}\label{defWD}
\defWD{n}{t} = \prod_{1 \leq i<j \leq n} (x_i-tx_j). 
\end{equation}
Note that $\defWD{n}{1} = \WD{n}$ and $\defWD{n}{0} = x^\rho$. 

\begin{Theorem}[Weyl Character Formula for $\text{GL}_n$]\label{schur}
The Schur polynomial corresponding to the partition $\l$ of length $n$ is 
\begin{equation}  s_\lambda(x)
= \sum_{\sigma \in S_n} \sigma \left(\frac{x^{\lambda + \rho}}{\WD{n}} \right)
. \end{equation}
\end{Theorem}

We may define the Hall-Littlewood polynomials analogously using the deformation of the Weyl denominator as follows.
\begin{Definition}
\label{HL}
The Hall-Littlewood polynomial for a partition $\l$ of length $n$ is
\begin{equation}
\HL{\lambda} = \sum_{\sigma \in S_n} \sigma \left( x^{\l} \frac{\defWD{n}{t}}{\WD{n}}\right)
\end{equation}
\end{Definition}
\noindent It is not difficult to see that $\HLtwo{\l}{x;0} = s_\lambda(x)$ and $\HLtwo{\lambda}{x;1} = \sum_{\sigma \in S_n} \sigma(x^\l)$, the monomial symmetric function $m_{\l}(x)$. 

Note: The Hall-Littlewood polynomials defined in \cite[Chapter III]{macdonaldbook} are given by
\begin{equation}
\MHL{\l}{x; t} = v_{\l}(t)\HL{\l},
\end{equation}
for a stabilizing factor $v_{\l}(t)$. Since the stabilizing factor may easily be multiplied to Theorem \ref{MainTheorem2*} if necessary, we choose to omit it in this paper, and refer to the polynomials $\HL{\l}$ as the Hall-Littlewood polynomials. 

\newpage

\begin{Definition}
A Gelfand-Tsetlin (GT) pattern is a triangular array of nonnegative integers of the form

\vspace{3mm}
\centerline{$\tau_{1,1} \;\;\;\;\;\;\; \tau_{1,2} \;\;\;\;\;\;\; \tau_{1,3} \;\;\;\;\;\;\; \ldots \;\;\;\;\;\;\; \tau_{1,n}$}

\centerline{$\,\tau_{2,2} \;\;\;\;\;\;\;\; \tau_{2,3} \;\;\;\;\;\;\;\; \ldots \;\;\;\;\;\;\;\; \tau_{2,n} $}

\centerline{ $ \ldots \;\;\;\;\;\;\;\; \ldots \;\;\;\;\;\;\;\; \ldots$}

\centerline{$\,\,\, \tau_{n-1,n-1} \;\; \tau_{n-1,n}$}

\centerline{$\,\,\,\,\tau_{n,n}$}
\vspace{3mm}
\noindent where each row $\row{i} = (\tau_{i,i},\tau_{i,i+1},\ldots,\tau_{i,n})$

is a \emph{weakly decreasing} partition, and two consecutive rows $r_i = (\tau_{i,i},\ldots,\tau_{i,n})$ and $r_{i+1} = (\tau_{i+1,i+1},\ldots,\tau_{i+1,n})$ satisfy the \emph{interleaving condition}:
\begin{equation} \tau_{i-1,j-1}\geq \tau_{i,j} \geq \tau_{i-1,j}. \end{equation}
\end{Definition}

For a partition $\a$, let $\G(\a)$ be the set of all GT patterns of top row $\a=r_1$. A \emph{strict} GT pattern is one in which each row $r_i$ is strictly decreasing. Given a partition $\a$, write $\GT(\a) \subseteq \G(\a)$ to be the set of all strict GT patterns with top row $\a$. 

\begin{Definition}[\cite{tokuyama}] \label{GTstats}
An entry $\tau_{i,j}$ in a GT pattern is 
\begin{itemize}
\item \emph{left-leaning} if $\tau_{i,j}=\tau_{i-1,j-1}$,
\item \emph{right-leaning} if $\tau_{i,j} = \tau_{i-1,j}$, and
\item \emph{special} if it is neither left-leaning nor right-leaning.
\end{itemize}

The quantities $l(T)$, $r(T)$, and $z(T)$  denote the number of left-leaning, right-leaning and special entries in a GT pattern respectively. 
\end{Definition}
\noindent Given a GT pattern with $n$ rows, define the statistic $m_i(T)$ as 
\begin{equation} 
m_i(T) = \begin{cases} 
|\row{i}|-|\row{i+1}| & \text{ for }\;1\leq i \leq n-1 \\ |\row{i}| & \text{ for }\;i=n \end{cases}, 
\end{equation}
and $m(T)$ as
\begin{equation} 
m(T)=\left(m_1(T),\ldots,m_n(T)\right).
\end{equation}

We now state the following theorem due to Tokuyama, which we generalize in the rest of this paper.

\begin{Theorem}[\cite{tokuyama}] \label{toks}
For any weakly decreasing partition $\lambda$ of length $n$, we have
\begin{equation} \label{tokseq}
\defWD{n}{q} \cdot s_\l(x) = \sum_{T\in \GT{(\lambda + \rho)}}  (1-q)^{z(T)}(-q)^{l(T)}x^{m(T)}.
\end{equation}

\end{Theorem}
\section{Additional Statistics on Gelfand-Tsetlin Patterns}\label{addstats}

To generalize Theorem \ref{toks} to the Hall-Littlewood polynomials, the previous statistics from Definition \ref{GTstats} of \cite{tokuyama} prove inadequate. Instead of only labelling each entry as left-leaning, right-leaning or special, we need to give each entry both a \emph{left-sided} property $p_l(\tau_{i,j})$ and a \emph{right-sided} property $p_r(\tau_{i,j})$. The left-sided property encodes the relationship the entry $\tau_{i,j}$ has to the entry directly above it and to its left, namely $\tau_{i-1,j-1}$. Similarly, the right-sided property encodes the relationship that $\tau_{i,j}$ has to $\tau_{i-1,j}$. 

\noindent The left-sided properties of an entry $p_l(\tau_{i,j})$ are assigned as
\begin{equation} \label{pl}
   p_l(\tau_{i,j}) = 
     \begin{cases}
       l \text{ (left)} & \text{if $\tau_{i,j} = \tau_{i-1,j-1}$}\\
       al \text{ (almost-left)} & \text{if $\tau_{i,j} = \tau_{i-1,j-1} - 1$} \\
       s \text{ (special)} & \text{otherwise}
     \end{cases},
\end{equation} 
and similarly, the right-sided properties of an entry $p_r(\tau_{i,j})$ are assigned as
\begin{equation} \label{pr}
p_r(\tau_{i,j}) = 
     \begin{cases}
       r \text{ (right)} & \text{if $\tau_{i,j} = \tau_{i-1,j}$}\\
       ar \text{ (almost-right)} & \text{if $\tau_{i,j} = \tau_{i-1,j} + 1$} \\
       s \text{ (special)} & \text{otherwise}
     \end{cases}.
\end{equation}

\begin{Definition} \label{c,d}
For an entry $\tau_{i,j}$ with $i>1$, we define

\begin{equation}
   c(\tau_{i,j}) =
     \begin{cases}
       0 & \text{if $p_l(\tau_{i,j}) = l$ or $p_r(\tau_{i,j})=r$}\\
       (1-t)(1-q) & otherwise \\
     \end{cases}
\end{equation} 
and for a property $p$, we define

\begin{equation}
   g(p) = 
     \begin{cases}
       -q & \text{if $p=l$} \\
       t & \text{if $p=al$} \\
       1 & \text{if $p=r$} \\
       -qt & \text{if $p=ar$} \\
       0 & \text{if $p=s$} \\
     \end{cases}
\end{equation}

\end{Definition}

With these we define two more functions: the first is a generalization of the expressions $(-q)$ and $(1-q)$ from Tokuyama's formula;  and the second considers the relation of an entry $\tau_{i,j}$ to the two entries above it in the GT pattern. 

\begin{Definition}\label{d,w}
For an entry $\tau_{i,j}$ with $i>1$, we define
\begin{equation}\label{def-w} w(\tau_{i,j}) = c(\tau_{i,j}) + g(p_l(\tau_{i,j})) + g(p_r(\tau_{i,j})).\end{equation}
For an entry $\tau_{i,j}$ with $i<j<n$, we define
\begin{equation}\label{def-d} d(\tau_{i,j})=g(p_r(\tau_{i+1,j}))\cdot g(p_l(\tau_{i+1,j+1})).\end{equation}
\end{Definition}

\begin{Example}\label{example:w,d}
Given the following segment of a GT pattern:

\centerline{$ 5 \; \; 3 \; \; 1 $}

\centerline{$4 \; \; 3 $}

\noindent We see that the $4$ has properties $al$ and $ar$. Thus $c(4) = (1-q)(1-t)$ and $w(4) = (1-q)(1-t) + t - qt = 1-q$. Similarly, we have $c(3) = 0$ and $w(3) = 0 - q + 0 = -q$. For the entries in the second row, we find $g(p_r(4)) = -qt$ and $g(p_l(3))= -q$, thus the 3 in the first row gives $d(3)=(-q t)\cdot (-q) = q^2t$. 
\end{Example}

For the reader's convenience, we provide Table \ref{table:wd} which lists all possible values for $w(\tau_{i,j})$ and $d(\tau_{i,j})$ that we may need to consider. One may notice that we omit the case for $w(\tau_{i,j})$ when $(p_l(\tau_{i,j}),p_r(\tau_{i,j})) = (l,r)$; this is simply because we need not consider this case at any point in our work.
\vspace{-1ex}
\begin{table}[h] 
\caption{Possible $w(\tau_{i,j})$ and $d(\tau_{i,j})$ values for an entry $\tau_{i,j}$.} \label{table:wd}
\centering
\begin{tabular}{|c|c|c|c|c|c|} \cline{1-2} \cline{4-6}
$(p_l(\tau_{i,j}),p_r(\tau_{i,j}))$ & $w(\tau_{i,j})$ & \hspace{2ex} & $p_r(\tau_{i+1,j})$ & $p_l(\tau_{i+1,j+1})$ & $d(\tau_{i,j})$ \\ \cline{1-2} \cline{4-6}

$(l,s)$ & $-q$ & \hspace{2ex} & 
    $s$ & $s$ & $0$ \\ \cline{1-2} \cline{4-6}
$(s,r)$ & $1$ & \hspace{2ex} & 
    $s$ & $l$, $al$ & $0$ \\ \cline{1-2} \cline{4-6}
$(l,ar)$ & $-q-qt$ & \hspace{2ex} & 
    $r$, $ar$ & $s$ & $0$\\ \cline{1-2} \cline{4-6}
$(al,r)$ & $1+t$ & \hspace{2ex} & 
    $r$ & $l$ & $-q$ \\ \cline{1-2} \cline{4-6}
$(s,ar)$ & $1-q-t$ & \hspace{2ex} &
    $ar$ & $l$ & $q^2t$ \\ \cline{1-2} \cline{4-6}
$(al,s)$ & $1-q+qt$ & \hspace{2ex} &
    $r$ & $al$ & $t$ \\ \cline{1-2} \cline{4-6}
$(al,ar)$ & $(1-q)$ & \hspace{2ex} & 
    $ar$ & $al$ & $-qt^2$ \\ \cline{1-2} \cline{4-6}
$(s,s)$ & $(1-q)(1-t)$ &  \multicolumn{3}{c}{} \\  \cline{1-2}
\end{tabular}
\end{table}

Example \ref{example:w,d} illustrates that to define $w(\tau_{i,j})$ and $d(\tau_{i,j})$, we only need to know two consecutive rows of a GT pattern. This leads us to the following definition. 
\begin{Definition} 
Suppose $\alpha$ is a strictly decreasing partition. We define $\GTtwo(\alpha)$ to be the set of all partitions $\mu$ such that the length of $\mu$ is one less than that of $\alpha$, and  $\alpha_i \geq \mu_i \geq \alpha_{i+1}$ for all $i$. For $\alpha$ of length 1, we let $\GTtwo(\alpha) = \{\emptyset\}$. 
\end{Definition}
\noindent This definition ensures that $\alpha$ and $\mu$ satisfy the interleaving condition, and so $\mu$ would be a valid weakly decreasing row directly below a row $\alpha$ in a GT pattern. Arranging $\alpha$ and $\mu$ in this manner, we are able to extend Definition \ref{d,w} to parts of $\mu$ and $\alpha$, and define $w(\mu_i)$ and $d(\alpha_i)$ for all appropriate $i$. 
\begin{Definition}\label{matrix}
Let $\alpha$ and $\mu$ be partitions with $\alpha$ strictly decreasing of length $n$ and $\mu \in \GTtwo(\alpha)$. Then we define
\begin{equation} 
\M{\alpha}{\mu} = \det
\begin{bmatrix}
w(\mu_1) & 1 & 0  & \; \ldots & 0 & 0\\
d(\alpha_2) & w(\mu_2) & 1  & \; \ldots & 0 & 0 \\
0 & d(\alpha_3) & w(\mu_3)  & \; \ldots & 0 & 0\\
\vdots & \vdots & \vdots & \ddots & \vdots &\vdots \\
0     & 0       & 0       &  \ldots & w(\mu_{n-2}) & 1 \\
0     & 0       & 0         & \ldots  & d(\alpha_{n-1})& w(\mu_{n-1}) \\
\end{bmatrix}. \end{equation}
\end{Definition}
\noindent If $\alpha$ is of length 1, and $\mu = \emptyset \in \GTtwo(\alpha)$, we define $\M{\alpha}{\mu} =1$. We also extend this notation to any pair $(\alpha; \mu)$ by defining $\M{\alpha}{\mu} = 0$ whenever $\mu \notin \GTtwo(\alpha)$.

\section{A Recursive Statement of Main Theorem}

We first notice that Tokuyama's formula can be restated as a summation over the set $\GTtwo(\a)$, as shown in \eqref{tokGT2}, where $\a = \lambda + \rho$ for some weakly decreasing partition $\lambda$.

Let $S_{\infty}$ denote the symmetric group on $\mathbb{N}$ and take $\zeta \in S_{\infty}$ to be the permutation which maps $k \mapsto k+1$ for each $k \in \mathbb{N}$. Then \eqref{toks} is equivalent to
\begin{equation}\label{tokGT2}
\defWD{n}{q} \cdot s_{\lambda}(x) = \sum_{\substack{\mu \in \GTtwo(\a) \\ \mu \text{ strict}}} (-q)^{l(\a;\mu)} (1-q)^{s(\a;\mu)} \ x_1^{|\a| - |\mu|} \ \zeta \left(\defWD{n-1}{q} \cdot s_{\mu-\rho}(x)\right).
\end{equation}
Here, the function $l(\a;\mu)$ is the number of `left-leaning' parts of $\mu$ with respect to $\a$, i.e. the number of parts $\mu_i$ which satisfy $\mu_i = \a_i$. The function $z(\a;\mu)$ is similarly defined with `special' parts. 

Re-expressing \eqref{tokseq} recursively as \eqref{tokGT2} motivates the search for a generalization of Tokuyama's formula expressed as a summation over the set $\GTtwo(\a)$, as below. This theorem is equivalent to Theorem \ref{MainTheorem2*}.

\begin{restatable}{thm}{Mainthree}\label{maintheorem}
Suppose $\lambda$ is a weakly decreasing partition of length $n$, and set $\a = \lambda + \rho$. Then, using the notation defined in Section \ref{addstats}, we have
\begin{equation} \label{maintheoremeq}
\defWD{n}{q} \cdot \HL{\lambda} = \sum_{\mu \in \GTtwo(\a)} \M{\a}{\mu} x_1^{|\a| - |\mu|} \ \zeta\left(\defWD{n-1}{q}\cdot \HL{\mu-\rho}\right). \end{equation}
\end{restatable}
\noindent It is worth noticing that unlike \eqref{tokGT2}, this expression requires all partitions in $\GTtwo(\a)$, including those that are non-strict, distinguishing it from Tokuyama's formula.

Theorem \ref{maintheorem} uses a determinant $\M{\a}{\mu}$ to determine the coefficient of the expression $x_1^{|\a| - |\mu|} \ \zeta\left(\defWD{n-1}{q}\cdot \HL{\mu-\rho}\right)$ in the relevant expansion of $\defWD{n}{q} \cdot \HL{\lambda}$. We use induction on the length of $\l$ and cofactor expansion of the determinant $\M{\a}{\mu}$ to prove Theorem \ref{maintheorem}. We omit the computations of the base cases in which $\l$ has length $1$ or $2$, which are easy to verify.

We move on to the general case of some weakly decreasing partition $\lambda$ of length $n>2$. For the remainder of this proof, we fix some general notation for partitions. Firstly, the partition $\lambda$ and its length $n$ are now fixed, and in any new notation to follow, these will remain independent of the variables in the expression. We consequently fix the partition $\a = \lambda + \rho$, also of length $n$. When referring to an arbitrary partition, we use $\kappa$ of length $m$. Finally, the partition $\mu$ will consistently be used as an arbitrary element of some set of the form $\GTtwo(\kappa)$.

Given a partition $\kappa = (\kappa_1,\ldots,\kappa_m)$, we will write
\begin{equation}
\exone{\kappa} = (\kappa_2,\ldots, \kappa_m) \text{ and }\extwo{\kappa} = (\kappa_3,\ldots,\kappa_m). 
\end{equation} 
Also, we define
\begin{equation}
\defdeltaone{i}{l}{t} := \prod_{\substack{1 \leq a \leq n \\ a \neq i}} x_l - tx_a, 
\end{equation}
and similarly
\begin{equation}
\defdeltatwo{i}{j}{l}{t} := \prod_{\substack{1 \leq a \leq n \\ a \neq i,j}} x_l - tx_a, \qquad \text{and} \qquad
\defdeltathree{i}{j}{k}{l}{t} := \prod_{\substack{1 \leq a \leq n \\ a \neq i,j,k}} x_l - tx_a.
\end{equation}
As with $\WD{n}$, we define $\deltaone{i}{l}=\defdeltaone{i}{l}{1}$ and similarly $\deltatwo{i}{j}{l}=\defdeltatwo{i}{j}{l}{1}$ and $\deltathree{i}{j}{k}{l} = \defdeltathree{i}{j}{k}{l}{1}$.

Our inductive hypotheses will be
\begin{equation}\label{indhyp1}
 \HL{\exone{\l}} \cdot \defdeltatwo{1}{n}{1}{q}   = \sum_{\mu \in \GTtwo(\exone{\a})} \M{\exone{\a}}{\mu} x_1^{|\exone{\a}| - |\mu|} \zeta\left( \HL{\mu-\rho} \right),
\end{equation}
and
\begin{equation}\label{indhyp2}
\HL{\extwo{\l}} \cdot \defdeltathree{1}{n-1}{n}{1}{q}   = \sum_{\mu \in \GTtwo(\extwo{\a})} \M{\extwo{\a}}{\mu} x_1^{|\extwo{\a}| - |\mu|} \zeta\left(\HL{\mu-\rho}\right).
\end{equation}

Note that multiplying both sides of \eqref{indhyp1} by $\zeta\left(\defWD{n-1}{q}\right)$ and \eqref{indhyp2} by $\zetatwo{1}{2}\left( \defWD{n-2}{q}\right)$ gives the form seen in Theorem \ref{maintheorem}.

Prior to the proof, we require another operator: We generalize the notion of $\zeta$ to $\zeta_i \in S_{\infty}$ which is defined to take $k \mapsto k+1$ for each $k \in \mathbb{N}$ with $k \geq i$. Furthermore, we also define $\zeta_{i,j} = \zeta_{j,i} = \zeta_j \zeta_i \in S_{\infty}$ when $i < j$. We notice that these operators act on a polynomial $f(x) = f(x_1, \ldots, x_m)$ to obtain 
\begin{equation}
\zeta_i(f(x))=f(x_1,\ldots,x_{i-1},x_{i+1},\ldots,x_{m+1}),
\end{equation}
and
\begin{equation}
\zeta_{i,j} (f(x)) = \zeta_{j,i} (f(x))=f(x_1, \ldots x_{i-1} , x_{i+1} , \ldots x_{j-1}, x_{j+1}, \ldots , x_{m+2}).
\end{equation}
Notice that $\zeta_1 = \zeta$, and we shall use the latter in the proof to follow.

We begin with a series of lemmas.


\begin{Lemma}\label{lambda12}
For an arbitrary partition $\kappa$ of length $m > 2$, we express $\HL{\kappa}$ recursively as
\begin{equation}\label{lambda1}
\HL{\kappa}=\sum_{1\leq i \leq m} x_i^{\kappa_1} \left(\defProdone{i}{i} \right) \zeta_i\left(\HL{\exone{\kappa}}\right),
\end{equation}
and 
\begin{equation}\label{lambda2}
\HL{\kappa}=\sum_{1\leq i \leq m} \sum_{\substack{1 \leq j \leq m \\ j \neq i}} x_i^{\kappa_1}x_j^{\kappa_2} \left( \defProdone{i}{i}  \defProdtwo{i}{j}{j} \right) \zeta_{i,j} \left(\HL{\extwo{\kappa}}\right).
\end{equation}
\end{Lemma}  
\begin{proof}
Let the permutation $\psi_i = (i\;i-1\;\cdots\;1)$, and let $H$ be the symmetric group acting on the $(n-1)$ indices $(2,3,\ldots,n)$. Then
\begin{align*}
\HL{\kappa}
& = \sum_{1 \leq i \leq m} \sum_{\sigma \in H} \psi_i \ \sigma \left(\frac{x^{{\kappa}} \defWD{m}{t}}{\WD{m}}\right)
= \sum_{1 \leq i \leq m} \sum_{\sigma \in H} \psi_i \ \sigma \left(x_1^{\kappa_1}\defProdone{1}{1} \;  \zeta\left(\frac{x^{\exone{\kappa}} \defWD{m-1}{t}}{\WD{m-1}}\right)\right).
\end{align*}
Because $\sigma \in H$ does not permute $x_1$, we have
\begin{align*}
\HL{\kappa} 
& =\sum_{1 \leq i \leq m} \psi_i \left(x_1^{\kappa_1}\defProdone{1}{1} \; \zeta\left(\HL{\exone{\kappa}}\right)\right) = \sum_{1\leq i \leq m} x_i^{\kappa_1} \left(\defProdone{i}{i} \right) \zeta_i\left(\HL{\exone{\kappa}}\right).
\end{align*}
We further obtain \eqref{lambda2} by applying \eqref{lambda1} to the $\HL{\exone{\kappa}}$ in \eqref{lambda1}. 
\end{proof}


\begin{Lemma}\label{originalproof}
Let $O_i= \big( (x_i - tx_1) x_i^{\lambda_1} - (x_1 -tx_i)x_1^{\lambda_1-\lambda_2}x_i^{\lambda_2} \big) / (x_i-tx_1)$. Then we have 
\begin{multline}\label{originalproofeq}
\HL{\lambda} = \sum_{2 \leq i \leq n} O_i \left( \defProdone{i}{i}\right)
\zeta_i\left(\HL{\exone{\lambda}}\right) \\ 
 -x_1^{\lambda_1-\lambda_2} \sum_{2 \leq i \leq n}\sum_{\substack{ 2 \leq j \leq n \\ j \neq i}}
tx_i^{\lambda_2}  x_j^{\lambda_2} \defProdtwo{1}{i}{i}\defProdthree{1}{i}{j}{j} \zeta_{i,j}\left(\HL{\extwo{\lambda}}\right).
\end{multline}

\end{Lemma}

\begin{proof}
We begin by observing that
\begin{equation*}
0=\sum_{2 \leq i \leq n} \sum_{2\leq j \leq n} (-1)^{i+j}x_i^{\lambda_2}  x_j^{\lambda_2}(x_i-x_j) \;  \zetatwo{i}{j} \left( \WD{n-2} \cdot \HL{\extwo{\lambda}} \right)
\defdeltatwo{1}{i}{i}{t} \defdeltatwo{1}{j}{j}{t}.
\end{equation*}
This can easily be seen by swapping the subscripts $i$ and $j$ in the right hand side, revealing RHS $= -$RHS.

We divide through the equality above by $\WD{n}$, altering the products and the bounds of the summation, and multiply by $x_1(1-t)$ to find 
\begin{equation*}
0 = \sum_{2 \leq i \leq n}\sum_{\substack{ 2 \leq j \leq n \\ j \neq i}} \frac{ x_1 (1-t)(x_j - tx_i)   x_i^{\lambda_2}  x_j^{\lambda_2}}{(x_j-x_1)(x_i-tx_1)}  \zeta_{i,j}(\HL{\extwo{\lambda}})
 \defProdone{i}{i} \defProdthree{1}{i}{j}{j} . 
\end{equation*}
Using the identity \begin{equation*}
\frac{x_1(1-t)(x_j-tx_i)}{(x_j-x_1)(x_i-tx_1)} = \frac{(x_1-tx_i)(x_j-tx_1)}{(x_i-tx_1)(x_j-x_1)} + t\frac{x_i-x_1}{x_i-tx_1},
\end{equation*}
we break the double summation into two parts; in particular, if we take
\begin{equation}\label{L}
L = \sum_{2 \leq i \leq n}\sum_{\substack{ 2 \leq j \leq n \\ j \neq i}}
tx_i^{\lambda_2}  x_j^{\lambda_2} \defProdtwo{1}{i}{i}\defProdthree{1}{i}{j}{j} \zeta_{i,j}\left(\HL{\extwo{\lambda}}\right),
\end{equation}
then
\begin{equation} \label{last}
0 = 
L + \sum_{2 \leq i \leq n}\sum_{\substack{ 2 \leq j \leq n \\ j \neq i}} 
\frac{(x_1-tx_i)x_i^{\lambda_2}  x_j^{\lambda_2}}{(x_i-tx_1)}\zeta_{i,j}(\HL{\extwo{\lambda}})
 \defProdone{i}{i} \defProdtwo{i}{j}{j}.
  \end{equation}
Now, one can see that
\begin{equation}\label{Weird2}
-x_1^{\lambda_2} \defProdone{1}{1} \zeta\left(\HL{\exone{\lambda}}\right) 
= x_1^{\lambda_2} \sum_{2 \leq i \leq n} \defProdone{i}{i} \defProdtwo{1}{i}{1} 
\frac{(x_1-tx_i)  x_i^{\lambda_2}}{(x_i-tx_1)}\zeta_{1,i}\left(\HL{\extwo{\lambda}}\right)
\end{equation}
holds by expressing $\HL{\exone{\lambda}}$ of the left hand side explicitly using \eqref{lambda1} and rearranging the result.  

We notice that the right hand side of \eqref{Weird2} is equivalent to setting $j=1$ in the double summation of \eqref{last}. Thus, adding either side of \eqref{Weird2} to either side of \eqref{last} respectively, we have
\begin{equation*}
 - x_1^{\lambda_2} \defProdone{1}{1}  \zeta\left(\HL{\exone{\lambda}}\right) 
= L + \sum_{2 \leq i \leq n}\sum_{\substack{ 1 \leq j \leq n \\ j \neq i}} 
 \frac{(x_1-tx_i)x_i^{\lambda_2}  x_j^{\lambda_2}}{(x_i-tx_1)}\zeta_{i,j}(\HL{\extwo{\lambda}})
 \defProdone{i}{i} \defProdtwo{i}{j}{j} . 
\end{equation*}
Multiplying through by  $-x_1^{\lambda_1-\lambda_2}$, and adding
\[\displaystyle \sum_{2 \leq i \leq n} x_i^{\lambda_1} \defProdtwo{i}{j}{j} \zeta_i(\HL{\exone{\lambda}}),\] to both sides, we may apply \eqref{lambda1} once on either side of the equality to obtain
\begin{equation*}
\HL{\lambda} =
 -x_1^{\lambda_1-\lambda_2} L + \sum_{2 \leq i \leq n} \left( \defProdone{i}{i}  \right) \left( x_i^{\lambda_1} - x_1^{\lambda_1-\lambda_2}x_i^{\lambda_2}\frac{x_1-tx_i}{x_i-tx_1} \right) 
\zeta_{i}(\HL{\exone{\lambda}}).
\end{equation*}
Recalling $O_i$ as
\begin{equation*}
O_i = \frac{(x_i - tx_1) x_i^{\lambda_1} -(x_1 -tx_i)x_1^{\lambda_1-\lambda_2}x_i^{\lambda_2} }{x_i-tx_1}
= x_i^{\lambda_1} - x_1^{\lambda_1-\lambda_2}x_i^{\lambda_2}\frac{x_1-tx_i}{x_i-tx_1} ,
\end{equation*}
we combine the two summations on the right hand side and recall $L$ from \eqref{L} to write
\begin{multline*}
\HL{\lambda} = \sum_{2 \leq i \leq n} O_i \left( \defProdone{i}{i}\right)
\zeta_i\left(\HL{\exone{\lambda}}\right) \\ 
 -x_1^{\lambda_1-\lambda_2} \sum_{2 \leq i \leq n}\sum_{\substack{ 2 \leq j \leq n \\ j \neq i}}
tx_i^{\lambda_2}  x_j^{\lambda_2} \defProdtwo{1}{i}{i}\defProdthree{1}{i}{j}{j} \zeta_{i,j}\left(\HL{\extwo{\lambda}}\right).
\end{multline*}\qedhere
\end{proof}

We introduce two related functions that will be used in the upcoming lemmas. For some non-negative integers $u$ and $v$, we define
\begin{equation}
F_{\exone{\lambda}}(u) : = \sum_{\mu \in \GTtwo(\exone{\a})} \M{\exone{\a}}{\mu} x_1^{|\exone{\a}| - |\mu|} \zeta\left(\HL{\conc{(u)}{\mu-\rho}}\right),
\end{equation} 
and
\begin{equation}
F_{\extwo{\lambda}}(u,v) : = \sum_{\mu \in \GTtwo(\extwo{\a})} \M{\extwo{\a}}{\mu} x_1^{|\extwo{\a}| - |\mu|} \zeta\left(\HL{\conc{(u,v)}{\mu-\rho}}\right).
\end{equation} 

\begin{Lemma}\label{GT12} 
Suppose $u$ and $v$ are some non-negative integers. Then, assuming the inductive hypotheses in \eqref{indhyp1} and \eqref{indhyp2}, we have
\begin{equation}\label{GT1}
F_{\exone{\lambda}}(u) = \sum_{2 \leq i \leq n} x_i^{u} \defProdtwo{1}{i}{i} \zeta_{i}\left(\HL{\exone{\lambda}}\right) \defdeltatwo{1}{i}{1}{q},
\end{equation}
and
\begin{equation}\label{GT2}
F_{\extwo{\lambda}}(u,v) = \sum_{2\leq i \leq n} \sum_{\substack{2 \leq j \leq n \\ j \neq i}} x_i^{u}x_j^{v} 
\defProdtwo{1}{i}{i}\defProdthree{1}{i}{j}{j} \zeta_{i,j}\left(\HL{\extwo{\lambda}}\right)  \defdeltathree{1}{i}{j}{1}{q},
\end{equation}
\end{Lemma}  

\begin{proof}
The proof of \eqref{GT2} is almost identical to that of \eqref{GT1}, apart from using \eqref{lambda2} and \eqref{indhyp2} in place of \eqref{lambda1} and \eqref{indhyp1}. For brevity, we only present the detailed proof of \eqref{GT1}.

By applying \eqref{lambda1} to write $\HL{ \conc{(u)}{(\mu-\rho)}}$ in terms of $u$ and $\HL{\mu-\rho}$, the left hand side of \eqref{GT1} becomes
\begin{equation*}
\sum_{\mu \in \GTtwo(\exone{\a})} \M{\exone{\a}}{\mu} x_1^{|\exone{\a}| - |\mu|} \; \zeta\left(\sum_{1\leq i \leq n-1} x_i^{u} \defProdtwo{i}{n}{i} \zeta_i\left(\HL{\mu-\rho}\right)\right).
\end{equation*}
We rearrange this expression to write this as
\begin{equation*}
\sum_{2 \leq i \leq n} x_i^{u} \defProdtwo{1}{i}{i} \zeta_i \left( \sum_{\mu \in \GTtwo(\exone{\a})} \M{\exone{\a}}{\mu} x_1^{|\exone{\a}| - |\mu|} \zeta\left(\HL{\mu-\rho}\right) \right).
\end{equation*}
Finally, we replace the argument of $\zeta_i$ using the inductive hypothesis in \eqref{indhyp1} to give the desired result.
\end{proof}

\begin{Lemma}\label{w}
Recall $O_i= \big( (x_i - tx_1) x_i^{\lambda_1} - (x_1 -tx_i)x_1^{\lambda_1-\lambda_2}x_i^{\lambda_2} \big) / (x_i-tx_1)$ from Lemma \ref{originalproof}. Then we have 
\begin{equation} 
\sum_{\mu \in \GTtwo(\a)} w(\mu_1)\M{\exone{\a}}{\exone{\mu}} x_1^{|\a| - |\mu|} \zeta\left(\HL{\mu-\rho}\right) 
 =    \sum_{2 \leq i \leq n} 
 O_i  \left( \defProdone{i}{i} \right)
 \zeta_i\left(\HL{\exone{\l}}\right) \defdeltaone{1}{1}{q}.
\end{equation}
\end{Lemma} 
\begin{proof}
First, notice that 
$(x_i-tx_1)O_i/(x_i-x_1) = Q_i/({x_1-qx_i})$,  where
\begin{align*}
Q_i & = -q x_i^{\lambda_1+1}     +tx_1 x_i^{\lambda_1}  -qtx_1^{\lambda_1-\lambda_2}   x_i^{\lambda_2+1}  \\ & \quad +x_1^{\lambda_1-\lambda_2+1}  x_i^{\lambda_2} +\sum_{\lambda_2 < i \leq \lambda_1}(1-q)(1-t)x_1^{\lambda_1+1-i}x_i^{i}.
\end{align*}
This can be shown through simple algebraic manipulation, considering three cases for $\lambda$, namely $(1)$: $\lambda_1 = \lambda_2$, $(2)$: $\lambda_1 = 1 + \lambda_2$ and $(3)$: $\lambda_1 > 1 + \lambda_2$. 

Substituting the claim in the right hand side of the lemma gives
\begin{equation*}
\text{RHS}  = \sum_{2 \leq i \leq n} \defProdtwo{1}{i}{i} \zeta_i\left(\HL{\exone{\lambda}}\right) Q_i \  \defdeltatwo{1}{i}{1}{q} .
\end{equation*}
Then, expanding $Q_i$ and applying \eqref{GT1}, we have
\begin{align*}
\text{RHS} &= -q \cdot F_{\exone{\lambda}}(\lambda_1+1) 
 + tx_1\cdot F_{\exone{\lambda}}(\lambda_1) 
 -qtx_1^{\lambda_1-\lambda_2} \cdot F_{\exone{\lambda}}(\lambda_2+1) \\
 & \qquad  +x_1^{\lambda_1-\lambda_2+1}\cdot F_{\exone{\lambda}}(\lambda_2)
 + \sum_{\lambda_2 < i \leq \lambda_1} (1-q)(1-t)x_1^{\lambda_1+1-i}\cdot F_{\exone{\lambda}}(i). 
\end{align*}
Examining Definition \ref{c,d}, we see that the first two coefficients above are precisely the nonzero possibilities of $g(p_l(\mu_1))$; the next two are precisely the nonzero possibilities of $g(p_r(\mu_1))$; and the final summation is over all the nonzero possibilities of $c(\mu_1)$. Recalling from Definition \ref{d,w} that $w(\mu_1)=c(\mu_1)+g(p_l(\mu_1))+g(p_r(\mu_1))$, we simply have
\begin{equation} 
\text{RHS } = \sum_{\mu \in \GTtwo(\a)} w(\mu_1)\M{\exone{\a}}{\exone{\mu}} x_1^{|\a| - |\mu|} \; \zeta\left(\HL{\mu - \rho}\right). \qedhere 
\end{equation} 
\end{proof}

\begin{Lemma}\label{d}
We have 
\begin{multline} 
\sum_{\mu \in \GTtwo(\a)} d(\a_2)\M{\extwo{\a}}{\extwo{\mu}} x_1^{|\a| - |\mu|} \; \zeta\left(\HL{\mu - \rho}\right) \\
= x_1^{\lambda_1-\lambda_2}\sum_{2 \leq i \leq n}\sum_{\substack{ 2 \leq j \leq n \\ j \neq i}}
tx_i^{\lambda_2}  x_j^{\lambda_2} \defProdtwo{1}{i}{i}\defProdthree{1}{i}{j}{j} \zeta_{i,j}\left(\HL{\extwo{\lambda}}\right) \defdeltaone{1}{1}{q},
\end{multline}
\end{Lemma}

\begin{proof}
Expanding the factor $(x_1-qx_i)(x_1-qx_j)$ from the product $ \defdeltaone{1}{1}{q}$, the right hand side of the above equality becomes
\hspace{1ex}
\begin{equation*}
\sum_{2\leq i \leq n} \sum_{\substack{2 \leq j \leq n \\ j \neq i}} tx_1^{\lambda_1-\lambda_2}x_i^{\lambda_2}x_j^{\lambda_2}  (x_1-qx_i)(x_1-qx_j) \defProdtwo{1}{i}{i}\defProdthree{1}{i}{j}{j} \zeta_{i,j}\left(\HL{\extwo{\lambda}}\right) \defdeltathree{1}{i}{j}{1}{q}. 
\end{equation*}

We see from Proposition \ref{shift} that $F_{\extwo{\lambda}}(\lambda_2,\lambda_2+1) = t\cdot F_{\extwo{\lambda}}(\lambda_2+1,\lambda_2)$. Then distributing $(q^2 x_i x_j  -q x_1 x_i -q x_1 x_j  + x_1^{2})$ over the summation and applying \eqref{GT2} yields 
\begin{align*}
\text{RHS} & =  q^2t x_1^{\lambda_1-\lambda_2} \cdot F_{\extwo{\lambda}}(\lambda_2+1,\lambda_2+1) -q x_1^{\lambda_1-\lambda_2+1} \cdot F_{\extwo{\lambda}}(\lambda_2,\lambda_2+1) \\ 
& \quad -qt^2 x_1^{\lambda_1-\lambda_2+1} \cdot F_{\extwo{\lambda}}(\lambda_2+1,\lambda_2) +t x_1^{\lambda_1-\lambda_2+2} \cdot F_{\extwo{\lambda}}(\lambda_2,\lambda_2,\mu-\rho) .\\
\intertext{Recalling from Definition \ref{c,d} that $d(\a_2)=g(p_r(\mu_1)) \cdot g(p_l(\mu_2))$, we notice that each of the four coefficients in the previous expression corresponds exactly to each of the four possible nonzero values for $d(\a_2)$. Thus, we have}
\text{RHS} & =  \sum_{\mu \in \GTtwo(\a)} d(\a_2)\M{\extwo{\a}}{\extwo{\mu}} x_1^{|\a| - |\mu|} \zeta\left(\HL{\mu - \rho}\right). \qedhere
\end{align*}
\end{proof}

We return to the proof of Theorem \ref{maintheorem}.

\begin{proof}
Lemma \ref{originalproof} gave us that
\begin{multline*}
\HL{\lambda} = \sum_{2 \leq i \leq n} O_i \left( \defProdone{i}{i}\right)
\zeta_i\left(\HL{\exone{\lambda}}\right) \\ 
 -x_1^{\lambda_1-\lambda_2} \sum_{2 \leq i \leq n}\sum_{\substack{ 2 \leq j \leq n \\ j \neq i}}
tx_i^{\lambda_2}  x_j^{\lambda_2} \defProdtwo{1}{i}{i}\defProdthree{1}{i}{j}{j} \zeta_{i,j}\left(\HL{\extwo{\lambda}}\right).
\end{multline*}
Furthermore, assuming the inductive hypotheses \eqref{indhyp1} and \eqref{indhyp2}, Lemmas \ref{w} and \ref{d} state that
\begin{equation*} 
\sum_{\mu \in \GTtwo(\a)} w(\mu_1)\M{\exone{\a}}{\exone{\mu}} x_1^{|\a| - |\mu|} \zeta\left(\HL{\mu-\rho}\right) 
 =    \sum_{2 \leq i \leq n} 
 O_i  \left( \defProdone{i}{i} \right)
 \zeta_i\left(\HL{\exone{\lambda}}\right) \defdeltaone{1}{1}{q}.
\end{equation*}
and
\begin{multline*} 
\sum_{\mu \in \GTtwo(\a)} d(\a_2)\M{\extwo{\a}}{\extwo{\mu}} x_1^{|\a| - |\mu|} \; \zeta\left(\HL{\mu-\rho}\right) \\
= x_1^{\lambda_1-\lambda_2}\sum_{2 \leq i \leq n}\sum_{\substack{ 2 \leq j \leq n \\ j \neq i}}
tx_i^{\lambda_2}  x_j^{\lambda_2} \defProdtwo{1}{i}{i}\defProdthree{1}{i}{j}{j} \zeta_{i,j}\left(\HL{\extwo{\lambda}}\right) \defdeltaone{1}{1}{q}.
\end{multline*}
Hence, it is clear that
\vspace{-1ex}
\begin{multline*}
\defdeltaone{1}{1}{q} \cdot \HL{\lambda} = \sum_{\mu \in \GTtwo(\a)} w(\mu_1)\M{\exone{\a}}{\exone{\mu}} x_1^{|\a| - |\mu|} \zeta(\HL{\mu-\rho})  \\
 - \sum_{\mu \in \GTtwo(\a)} d(\a_2)\M{\extwo{\a}}{\extwo{\mu}} x_1^{|\a| - |\mu|} \zeta(\HL{\mu-\rho}) .
\end{multline*}
Finally, recalling that $\M{\a}{\mu} = w(\mu_1)\M{\exone{\a}}{\exone{\mu}} - d(\a_2)\M{\extwo{\a}}{\extwo{\mu}}$ and observing that $\defWD{n}{q} = \defdeltaone{1}{1}{q} \cdot \zeta(\defWD{n-1}{q})$, we multiply by $\zeta(\defWD{n-1}{q})$ to conclude
\begin{equation*}
\defWD{n}{q} \cdot \HL{\lambda} = \sum_{\mu \in \GTtwo(\a)} \M{\a}{\mu} x_1^{|\a| - |\mu|} \zeta(\WD{n-1}(q)\cdot \HL{\mu-\rho}). \qedhere
\end{equation*}
\end{proof}

\section{Weakly Decreasing Partitions and Raising Operators}

Theorem \ref{maintheorem} expresses the Hall-Littlewood polynomial recursively in terms of Hall-Littlewood polynomials in one fewer variables. 
The partitions $\mu \in \GTtwo(\a)$ indexing these polynomials are guaranteed to be weakly decreasing by the interleaving condition, but they are not necessarily strictly decreasing.

To express $\mu \in \GTtwo(\a)$ in terms of strictly decreasing partitions, we relate the Hall-Littlewood polynomial associated to a weakly decreasing partition to one associated to a specific strictly decreasing partition, which is related to the weakly decreasing partition through a specified sequence of Young's raising operators.

\begin{Definition}
A \emph{raising operator} $\R$ is a product of operations $[i\;j]$ with $i\leq j$ acting on some finite tuple $\l$ of nonnegative integers such that
\begin{equation}
[i\;j] \cdot (\lambda_1,\ldots, \lambda_n)=(\lambda_1,\ldots,\lambda_i-1,\ldots,\lambda_j+1,\ldots,\lambda_n).
\end{equation} 
\end{Definition} 
\noindent (Note that these are the inverses of Young's raising operators as defined in \cite[Chapter I]{macdonaldbook}.)

The length of a raising operator $\R$, denoted $l(\R)$, is defined as the number of operators in the minimal decomposition of $\R$ into elementary operators of the form $[i \; i+1]$. The identity raising operator Id acts trivially on the partition and is assigned length zero. 

\begin{Definition}\label{thetaset}
Given a strictly decreasing partition $\a$ of length $n$, we recursively define $\thetaset(\a)$ to be the set of raising operators such that
\begin{itemize}
\item The identity raising operator $\text{Id }\in \thetaset(\a)$, and
\item for all raising operators $\R \in \thetaset(\a)$, if the tuple $\R(\a) = (\a_1', \ldots, \a_n')$ contains consecutive parts $\a_i'$ and $\a_{i+1}'$ such that $\a_i' = \a_{i+1}'+ 2$, then $[i \; i+1] \cdot \R  \in \thetaset(\a)$.  
\end{itemize} 
\end{Definition}

\begin{Example}
For the partition $\a=(6,4,3,1)$, we have the set
\[\thetaset(\a) = \{\text{Id},[1\;2],[1\;3],[2\;4],[3\;4],[1\;2][3\;4] \}.\]
\end{Example}

\begin{Prop}\label{shift} 
Suppose $\l$ is some weakly decreasing partition and $\a = \l + \r$. Then, for each $\R\in \thetaset(\a)$, the following identity holds:
\begin{equation}\label{shifteq} \HL{\R(\lambda)}= t^{l(\R)} \cdot \HL{\lambda}. \end{equation}
\end{Prop} 

\begin{proof}
It is clear that \eqref{shifteq} holds for $\R = $ Id. Therefore, by the recursive definition of $\thetaset$, it suffices to prove that for an arbitrary tuple $\l = (\l_1, \ldots, \l_n)$ and $\a = \l + \r$ such that $\a_i = \a_{i+1} + 2$, or equivalently $\l_i = \l_{i+1} + 1$, we have

\begin{equation}\label{shifteqn} 
\HL{[i\;i+1](\lambda)}= t \cdot \HL{\lambda}. 
\end{equation} 

To prove \eqref{shifteqn}, let $g(x) = \defWD{n}{t}/(x_i - t x_{i+1})$ and $f(x)=x_1^{\lambda_1}\cdots x_{i-1}^{\lambda_{i-1}} \cdot x_{i+2}^{\lambda_{i+2}}\cdots x_n^{\lambda_n}$, noting that both $g(x)$ and $f(x)$ are invariant under the permutation $(i \; i+1)$. Then, if we take $\lambda_{i+1} = a$ and $\lambda_i = a+1$ for some integer $a$, and let $\sgn \sigma$ be the standard sign function for a permutation $\sigma \in S_n$, we see that
\begin{equation*}
\sum_{\sigma \in S_n}\left(\text{sgn }\sigma\right) \sigma\left(x^\lambda \defWD{n}{t}\right) 
= \sum_{\sigma \in S_n}\left(\text{sgn }\sigma\right) \sigma\left(f(x)g(x)x_i^{a+2} x_{i+1}^{a} - tf(x)g(x)x_i^{a+1} x_{i+1}^{a+1}\right).
\end{equation*} 
If some polynomial $h(x) = (i \; j)(h(x))$ for $(i \; j) \in S_n$, then it is known that $\sum_{\sigma \in S_n}\left(\text{sgn }\sigma\right) \sigma\left(h(x)\right)=0$.

Therefore, since  $tf(x)g(x)x_i^{a+1} x_{i+1}^{a+1}$ is invariant under the permutation $(i \; i+1)$, we have
\begin{equation}\label{unraisedWD}
\sum_{\sigma \in S_n}\left(\text{sgn }\sigma\right) \sigma\left(x^\lambda \defWD{n}{t}\right) = \sum_{\sigma \in S_n}\left(\text{sgn }\sigma\right) \sigma\left(f(x)g(x)x_i^{a+2} x_{i+1}^{a}\right).
\end{equation} 
Similarly, we can find that

\begin{align}
\sum_{\sigma \in S_n}\left(\text{sgn }\sigma\right) \sigma\left(x^{[i\;i+1](\lambda)}\defWD{n}{t}\right) 
& = (- t) \cdot \sum_{\sigma \in S_n}\left(\text{sgn }\sigma\right) \sigma\left(f(x)g(x)x_i^{a} x_{i+1}^{a+2}\right) \nonumber \\ \label{raisedWD}
& = t \cdot \sum_{\sigma \in S_n}\left(\text{sgn }\sigma\right) \sigma\left(f(x)g(x)x_i^{a+2} x_{i+1}^{a}\right), 
\end{align} 

and combining \eqref{unraisedWD} and \eqref{raisedWD} returns
\begin{equation*}
\sum_{\sigma \in S_n}\left(\text{sgn }\sigma\right) \sigma\left(x^{[i\;i+1](\lambda)}\defWD{n}{t}\right)= t\cdot\sum_{\sigma \in S_n}\left(\text{sgn }\sigma\right) \sigma\left(x^{\lambda}\defWD{n}{t}\right).
\end{equation*}
Dividing through by $\WD{n}$ gives us \eqref{shifteqn}, as desired.
\end{proof}

\section{A Deformation of Tokuyama's Formula}
By enabling us to express Hall-Littlewood polynomials $\HL{\mu-\rho}$ with nonstrict $\mu$ in terms of those with strict $\mu$, Proposition \ref{shift} allows us state Theorem \ref{maintheorem} to non-recursively as Theorem \ref{MainTheorem2*}, providing a deformation of Tokuyama's formula.

\begin{restatable}[]{thm}{Maintwo}
\label{MainTheorem2*}
Suppose $\l$ is a weakly decreasing partition of length $n$. Then, the product $\defWD{n}{q} \cdot \HL{\lambda}$ of the deformed Weyl denominator and the Hall-Littlewood polynomial can be expressed as a summation over the set $\GT(\l+\r)$ of all strict Gelfand-Tsetlin patterns of top row $\l+\r$ by
\begin{equation}\label{MainTheorem2eq}
\defWD{n}{q}\cdot \HL{\lambda} = \sum_{T \in \GT(\lambda + \r)} \prod_{i=1}^{n-1}\left(\sum_{\R \in \thetaset(\row{i+1})}t^{l(\R)}\M{\row{i}}{\R(\row{i+1})}\right)x^{m(T)}.
\end{equation} 
Here, the determinant $\M{r_i}{\R(r_{i+1})}$ and the set $\thetaset(r_{i+1})$ are defined in Definitions \ref{matrix} and \ref{thetaset} respectively.
\end{restatable}

\begin{proof}
We use induction, where our inductive hypothesis is that \eqref{MainTheorem2eq} holds for all Hall-Littlewood polynomials in $(n-1)$ variables. The base case formula of one variable is easy to check.

We now prove that \eqref{MainTheorem2eq} holds for a Hall-Littlewood polynomial corresponding to an arbitrary weakly decreasing partition $\lambda$ of length $n>1$, assuming the inductive hypothesis. Let $\a = \l+\r$. 
We notice that, while $\a$ must be strictly decreasing, a partition $\mu \in \GTtwo(\a)$ may be weakly decreasing.
If $\mu$ is not strictly decreasing, the computed Hall-Littlewood polynomial $\HL{\mu-\rho}$  has an increasing partition $\mu - \rho$. Proposition \ref{shift} allows us to express these Hall-Littlewood polynomials in terms of some Hall-Littlewood polynomials that are obtained from strictly decreasing $\tilde \mu \in \GTtwo(\a)$. In particular, for each non-strict $\mu \in \GTtwo(\a)$, there exists a unique strict $\tilde \mu \in \GTtwo(\a)$ and a unique element $\R \in \thetaset(\tilde \mu)$ such that $\R(\tilde \mu) = \mu$. Furthermore, for each $\R \in \thetaset(\tilde \mu)$, the tuple $\R(\tilde \mu)$ will either be a valid element of $\GTtwo(\a)$ or will cause $\M{\a}{\R(\tilde \mu)}=0$ and can be neglected. Hence, we may rewrite \eqref{maintheoremeq} of Theorem \ref{maintheorem} as an equivalent summation over strictly decreasing partitions in the set $\GTtwo(\a)$:
\begin{equation}\label{recurs}
\HL{\l} \cdot \defWD{n}{q}  =  \sum_{\substack{\mu \in \GTtwo(\a) \\ \mu \text{ strict}}} \sum_{\R \in \thetaset(\mu)} t^{l(\R)}\M{\a}{\R(\mu)} x_1^{|\a| - |\mu|} \zeta(\defWD{n-1}{q} \cdot \HL{\mu-\rho}).
\end{equation}
By applying the inductive hypothesis to the Hall-Littlewood polynomials $\HL{\mu-\rho}$ of $(n-1)$ variables in \eqref{recurs}, we are reduced to \eqref{MainTheorem2eq}.
\end{proof}

\begin{Example}
We present an example of Theorem \ref{MainTheorem2*} by computing the term obtained from a particular GT pattern. We take $\lambda=(1,0,0)$, so $\a = (3,1,0)$, and we compute the coefficient obtained from the GT pattern

\vspace{2mm}

\centerline{ $ 3 \;\;\;\;\;\;\;\; 1 \;\;\;\;\;\;\;\; 0$}

\centerline{$\,\,\, 2 \;\;\;\;\;\;\;\;  0$}

\centerline{$\,\,\,\,1$}

\vspace{2mm}

Having fixed a particular $T \in \GT(3,1,0)$, we iterate $i$, starting with $i=1$. Then we have the set $\thetaset(\row{2})= \{\text{Id},[1 \; 2]\} $. We denote $[1 \; 2]$ as $\R$. Then $\R(\row{2}) = (1,1)$ and $l(\R) = 1$. 

We thus have two possibilities to consider: one for each raising operator. In each case, we display the relevant patterns and compute the determinant of $M$. To minimize confusion, we have subscripts for integers that appear multiple times in a pattern. 

\noindent Below is the pattern associated to the raising operator $\text{Id}$:

\centerline{$3 \;\;\;\;\;\;\;\; 1 \;\;\;\;\;\;\;\; 0_1$}

\centerline{ $2 \;\;\;\;\;\;\;\; 0_2$} 

\noindent These rows give the matrix and corresponding coefficient: 
\begin{equation}\label{row21}
\left|
\begin{pmatrix}
w(2) & 1 \\
d(1) & w(0_2)   \\
\end{pmatrix} 
\right|
=
\left\vert \begin{matrix}
1-q & 1 \\
-qt^2 & 1+t   \\
\end{matrix}\right\vert = 1-q+t-qt+qt^2. 
\end{equation}

\noindent Below is the pattern associated to the raising operator $\R$:

\centerline{$3 \;\;\;\;\;\;\;\; 1_1 \;\;\;\;\;\;\;\; 0$}

\centerline{ $1_2 \;\;\;\;\;\;\;\; 1_3$}
\noindent As $l(\R_1)=1$ we multiply the determinant of the associated matrix by $t^1$ to obtain
\begin{equation}\label{row22}
t \cdot \left|
\begin{pmatrix}
w(1_2) & 1 \\
d(1_1) & w(1_3)   \\
\end{pmatrix} 
\right|
=
t \cdot \left\vert \begin{matrix}
1 & 1 \\
-q & -q-qt   \\
\end{matrix}\right\vert = -qt^2.  
\end{equation}

That concludes the consideration of possible second rows. We add all of the coefficients from each case of a second row, i.e. \eqref{row21} and \eqref{row22}, to find the total coefficient of 
\begin{equation} \label{row2}
1-q+t-qt+qt^2-qt^2= (1-q) (1 + t). 
\end{equation} 

Now iterating to $i=2$, we have that $\thetaset(\row{3})$ contains only the identity.

\centerline{$\,\,\, 2 \;\;\;\;\;\;\;\;  0$}

\centerline{$\,\,\,\,1$}

\noindent The determinant is simply
\begin{equation}\label{row4}
|w(1)| = (1-q). 
\end{equation} 

Finally, we take the product of all the coefficients we got from each row, i.e. \eqref{row2} and \eqref{row4}, and multiply this with $x^{m(T)}$ where $m(T) = (2,1,1)$. Thus, by Theorem \ref{MainTheorem2*}, the GT pattern contributes the monomial
\begin{equation}\label{exampletotal}
(1-q)^2 (1 + t) x_1^2x_2x_3,
\end{equation}
to the summation. As this is the unique GT pattern with top row $\a=(2,1,0)$ and $m(T)=(2,1,1)$, we should find that \eqref{exampletotal} gives the coefficient of $x_1^2x_2x_3$ in the expansion of $\defWD{3}{q}\cdot \HL{(1,0,0)}$. The reader can verify that this is indeed the case. 

\end{Example}
\section{Specializations} 
 
The results of this paper generalize Tokuyama's formula and several other existing results. We demonstrate a few of these specializations. 

\subsection{Tokuyama's formula.}

Recall from Definition \ref{HL} that $\HLtwo{\l}{x;0} = s_\l(x)$. We know that for all raising operators $\R \neq \text{Id}$, the length of $\R$ is at least $1$. Therefore, setting $t=0$ reduces Theorem \ref{MainTheorem2*} to 
\begin{equation*}
\defWD{n}{q}\cdot s_{\lambda} = \sum_{T \in \GT(\a)} \prod_{i=1}^{n-1}\M{\row{i}}{\row{i+1}}x^{m(T)}.
\end{equation*}
As these are all the identity cases on strict Gelfand-Tsetlin patterns, every row is strictly decreasing. This implies that consecutive entries cannot have $p_r(\tau_{i,j}) = r$ and $p_l(\tau_{i,j+1}) = l$, and consequently $d(\tau_{i,j})$ cannot be $-q$. All the remaining possibilities of $d(\tau_{i,j})$ are reduced to $0$ when $t=0$. Thus, if we let $\row{i+1} = (\mu_1, \ldots, \mu_{n-i})$, the matrix $M(\row{i};\row{i+1})$ simplifies to
 \begin{equation*}
 M(\row{i};\row{i+1}) = 
 \begin{pmatrix}
 w(\mu_1) & 1 &   \; \ldots &  0\\
 0 & w(\mu_2) &   \; \ldots &  0 \\
 \vdots & \vdots &  \ddots & \vdots \\
 0     & 0             & \ldots  & w(\mu_{n-i}) \\
 \end{pmatrix}
 \end{equation*}
Therefore, we have 
 \begin{equation}\label{toksimp} 
 \M{\row{i}}{\row{i+1}}=\prod_{k=1}^{n-i} w(\mu_k).\end{equation}
 Finally, returning to Tokuyama's terminology of left-leaning, right-leaning and special entries from Definition \ref{GTstats}, we find that $w(\tau_{i,j})$ simplifies to
  \begin{equation}
    w(\tau_{i,j}) = 
      \begin{cases}
        -q & \text{if $\tau_{i,j}$ is left-leaning}\\
        1-q & \text{if $\tau_{i,j}$ is special} \\
        1 & \text{if $\tau_{i,j}$ is right-leaning}
      \end{cases},
 \end{equation}
 and, substituting this into \eqref{toksimp} from earlier, we conclude with Tokuyama's formula:
 \begin{equation*}
\defWD{n}{q}\cdot s_{\lambda} = \sum_{T \in \GT(\a)} (-q)^{l(T)}(1-q)^{z(T)} x^{m(T)}.
\end{equation*}

Comparing the results of this paper with Tokuyama's formula reveals some interesting distinctions regarding the structure of the Hall-Littlewood polynomials in relation to the Schur polynomials. Theorem \ref{maintheorem} demonstrates that when expressing $\HL{\lambda}$ recursively in terms of $\HL{\mu-\rho}$,  it is more natural to include several non-strictly decreasing partitions $\mu$  in the summation.  This is not an issue for Tokuyama's formula as Schur polynomials of such non-strictly decreasing partitions $\mu$ are just $s_{\mu-\rho}(x) = 0$. 

 In fact, whilst Theorem \ref{MainTheorem2*} is stated as summations over strict GT patterns, the use of $\thetaset$ is  to allow an implicit summation over all possible non-strict rows. We thus naturally seek a way to consider the contributions of such rows directly, eliminating the more \emph{ad hoc} use of $\thetaset$. Both theorems also highlight the added complexity in a Hall-Littlewood polynomial as they account for the ordering among the entries in a GT pattern instead of simply counting entries as Tokuyama does with $z(T)$ and $l(T)$. 

\subsection{Stanley's formula.}
In \cite{stanley}, Stanley gave a formula for the Hall-Littlewood polynomials at the singular value $t=-1$, also known as the Schur $q$-polynomials, as a generating function of strict GT patterns of top row $\l$:
\begin{equation}\label{stanley's}
\HLtwo{\l}{x;-1} = \sum_{T \in \GT(\l)}2^{z(T)}x^{m(T)}
\end{equation}
Tokuyama subsequently showed in \cite{tokuyama} that his formula yields \eqref{stanley's} when the deformation parameter $q$ is set to $-1$. Theorem \ref{MainTheorem2*} thus specializes to \eqref{stanley's} at $t=0$ and $q=-1$, by virtue of specializing to Tokuyama's result.
However, setting $t$ to $-1$ in Theorem \ref{MainTheorem2*} also gives a deformation along $q$ of \eqref{stanley's}, and we can show that this deformation reduces to \eqref{stanley's} at $q=0$.

If $\tau_{i,j}$ is an entry in a GT pattern, then we call $\tau_{i,j}$ a $p$ entry (where $p$ is a property) if either $p_l(\tau_{i,j})$ or $p_r(\tau_{i,j})=p$. 

With this terminology, examining Theorem \ref{MainTheorem2*} at the singular values of $t=-1$ and $q=0$, we see that any pattern containing an $l$ entry gives an overall coefficient of zero, which is evident through cofactor expansion of $\M{\row{i}}{\row{i+1}}$. 

Thus we may simply sum over the set $\GT l(\a) \subset \GT(\a)$ that contains all GT patterns without $l$ entries. Furthermore, any non-trivial element $\phi \in \thetaset(r_i)$ will either cause an element of $\phi(r_i)$ to be an $l$ entry with respect to $r_{i+1}$, or result in $\phi(r_i) \notin \GTtwo(r_{i-1})$, both resulting in an overall coefficient of zero. Furthermore, we show that GT patterns with consecutive $r$ then $al$ entries give coefficient $0$.

Let $D=\prod_{i=2}^{n}\vert M(\row{i-1};\row{i})\vert$. Assume for the sake of contradiction that there is a GT pattern with consecutive $r$ then $al$ entries for which $D \neq 0$. Let $\tau_{i,j}$ and $\tau_{i,j+1}$ be the lowest consecutive $r$ then $al$ entries (there may be others on the same row, but none below). Then the $r$ entry $\tau_{i,j} = u$ for some $u \in \mathbb{N}$, and the $al$ entry $\tau_{i,j+1} = u-1$. Now consider the entry $\tau_{i+1,j+1}$. Since it cannot be $l$, or else $D=0$, we must have $\tau_{i+1,j+1}=u-1$. Then $w(\tau_{i+1,j+1}) = 0$; so to ensure $M(\row{i};\row{i+1})\neq 0$ (and consequently that $D \neq 0$), we must have that either $p_r(\tau_{i+1,j})=r$ or $p_l(\tau_{i+1,j+2})=al$. But this contradicts our hypothesis. Therefore, any GT pattern containing consecutive $r$ then $al$ entries yields an overall coefficient of zero.

Let the set $\GT l^*(\a) \subset \GT l(\a)$ contain all GT patterns without $l$ entries or consecutive $r$ then $al$ entries. For an entry $\tau_{i,j}$ in a strict GT pattern $T$, recall that $q$ divides $d(\tau_{i,j})$, and hence $d(\tau_{i,j})=0$ at $q=0$, unless $p_r(\tau_{i+1,j})=r$ and $p_l(\tau_{i+1,j+1})=al$. If $T \in \GT l^*(\a)$, then this cannot occur, so $d(\tau_{i,j})=0$ for all entries in any GT pattern $T \in \GT l^*(\a)$. Thus Theorem \ref{MainTheorem2*} simplifies to

\begin{equation}\label{lar}
\defWD{n}{0}\cdot \HLtwo{\l}{x;-1} = \sum_{T \in \GT l^*(\a)} \prod w(\tau_{i,j}) x^{m(T)},
\end{equation}
 where the product is taken over all possible $\tau_{i,j}$. 
 This leads us to introduce a bijective mapping, similar to one used in \cite{tokuyama}: 
\begin{equation}
\begin{array}{ll}
     \;     & \GT l^*(\a) \longrightarrow \GT(\l) \\
\tokbij :   &        \tau_{i,j} \longmapsto \tau_{i,j}+j-n \\
     \;     &     m(T) \longmapsto m(T)-\rho 
\end{array}.
\end{equation}
After applying this mapping, we have that $w(\tau_{i,j})=2$ for all \emph{special} entries and $w(\tau_{i,j}) = 1$ otherwise, and thus \eqref{lar} reduces to 
\begin{equation*}
x^\r\cdot \HLtwo{\l}{x;-1} = x^\r\cdot \sum_{T \in \GT(\lambda)} 2^{z(T)} x^{m(T)},
\end{equation*}
Dividing out by $x^\r$, we obtain \eqref{stanley's}.

\subsection{Monomial symmetric function.}

Recall from Definition \ref{HL} that the monomial symmetric function 
\begin{equation} \label{monomial}
m_\l(x) = \sum_{\sigma \in S_n} \sigma(x^\l)
\end{equation}
We note that, since $\HLtwo{\l}{x;1}=m_\l(x)$, we may obtain a new $q$-deformation of \eqref{monomial} by setting $t=1$ in Theorem \ref{MainTheorem2*}, although the result does not appear as elegant as the $t=0$ and $t=-1$ cases. This specialization is, however, not difficult to prove independently.

\section{Acknowledgments}
The authors thank Paul Gunnells for suggesting this area of research and for his invaluable guidance, and are likewise grateful to Lucia Mocz for her dedicated mentoring. In addition, they would like to thank Glenn Stevens and the PROMYS program, without which this research could not have been pursued. 
\newpage

\bibliographystyle{plain}
\setlength{\itemsep}{2ex}\normalsize
\bibliography{s9_references}





\end{document}